\documentclass[10pt]{article}
\usepackage{amsfonts}
\usepackage{latexsym}
\usepackage{cite}
\usepackage{amsmath,amsfonts,latexsym,amssymb}
\usepackage[mathscr]{eucal}
\usepackage{cases}
\usepackage{amsthm}

\usepackage[bf,small]{caption2}
\usepackage{float}
\usepackage{graphicx,color}
\usepackage{amsmath}
\usepackage{amssymb}
\usepackage[all]{xy}

\newtheorem{theorem}{theorem}[section]
\newtheorem{thm}[theorem]{Theorem}
\newtheorem{lem}[theorem]{Lemma}

\newtheorem{prop}[theorem]{Proposition}

\newtheorem{rmk}[theorem]{Remark}
\newtheorem{nota}[theorem]{Notation}

\begin{document}

\title{\vspace{-2cm}\textbf{Regular $t$-balanced Cayley maps on split metacyclic $2$-groups}}
\author{\Large Haimiao Chen \hspace{10mm} Jingrui Zhang}

\date{}
\maketitle

\begin{abstract}
  A regular $t$-balanced Cayley map on a group $\Gamma$ is an embedding of a Cayley graph on $\Gamma$ into a surface with certain special symmetric properties. We completely classify regular $t$-balanced Cayley maps for a class of split metacyclic $2$-groups.

  \bigskip
  \noindent {\bf Keywords:} regular Cayley map; $t$-balanced; split metacyclic $2$-group; reduction. \\
  {\bf MSC2020:} 05C25, 05C10, 20B25.
\end{abstract}

\section{Introduction}

Suppose $\Gamma$ is a finite group and $\Omega$ is a generating set of $\Gamma$ such that $\omega^{-1}\in\Omega$ whenever $\omega\in\Omega$, and the identity $1\notin\Omega$. The \emph{Cayley graph} ${\rm Cay}(\Gamma,\Omega)$
is the graph having the vertex set $\Gamma$ and the arc set $\Gamma\times\Omega$, where for $\eta\in \Gamma, \omega\in\Omega$, the arc from $\eta$ to $\eta\omega$ is denoted as $(\eta,\omega)$.

A cyclic permutation $\rho$ on $\Omega$ canonically induces a permutation on the arc set via $(\eta,\omega)\mapsto(\eta,\rho(\omega))$, and this equips each vertex $\eta$ with a ``cyclic order", which means a cyclic permutation on the set of arcs emanating from $\eta$.
This determines an embedding of ${\rm Cay}(\Gamma,\Omega)$ into a closed oriented surface, which is characterized by the property that each connected component of the complement of the Cayley graph is a disk. Such an embedding is called a {\it Cayley map} and denoted by $\mathcal{CM}(\Gamma,\Omega,\rho)$.
An {\it isomorphism} of Cayley maps $\mathcal{CM}(\Gamma,\Omega,\rho)\to\mathcal{CM}(\Gamma',\Omega',\rho')$ is by definition an isomorphism
${\rm Cay}(\Gamma,\Omega)\to{\rm Cay}(\Gamma',\Omega')$ which can be extended to an orientation-preserving homeomorphism between their embedding surfaces.

A Cayley map is called {\it regular} if its automorphism group acts regularly on the arc set, i.e., for any two arcs, there exists an automorphism sending one arc to the other. It was shown in \cite{JS02} that $\mathcal{CM}(\Gamma,\Omega,\rho)$ is regular if and only if there exist a {\it skew-morphism} which is a bijective function $\varphi:\Gamma\to\Gamma$, and a {\it power function} $\pi:\Gamma\to\{1,\ldots,\#\Omega\}$ (where $\#\Omega$ is the cardinality of $\Omega$), such that $\varphi|_{\Omega}=\rho$, $\varphi(1)=1$ and $\varphi(\eta\mu)=\varphi(\eta)\varphi^{\pi(\eta)}(\mu)$ for all $\eta,\mu\in\Gamma.$

Let $d=\#\Omega$, and $t$ be an integer with $t^2\equiv 1\pmod{d}$. A regular Cayley map $\mathcal{CM}(\Gamma,\Omega,\rho)$ is called {\it $t$-balanced} if
\begin{align}
\rho(\omega^{-1})=(\rho^{t}(\omega))^{-1} \qquad \text{for\ all\ } \omega\in\Omega;  \label{eq:t-balance}
\end{align}
in particular, it is called {\it balanced} if $t\equiv 1\pmod{d}$ and {\it anti-balanced} if $t\equiv-1\pmod{d}$.
It is the residue modulo $d$ rather than $t$ itself, that plays a key role.
From now on we assume $0<t<d$, and abbreviate ``regular $t$-balanced Cayley map" to ``RBCM$_{t}$".

Recall some facts on RBCM$_t$ from \cite{Ch17} Proposition 1.2.
\begin{prop}  \label{prop:RBCMt}
{\rm(a)} A Cayley map $\mathcal{CM}(\Gamma,\Omega,\rho)$ is a RBCM$_{1}$ if and only if $\rho$ can be extended to an automorphism of $\Gamma$.

{\rm(b)} Suppose $t>1$. A Cayley map $\mathcal{CM}(\Gamma,\Omega,\rho)$ is a RBCM$_{t}$ if and only if $\rho$ can be extended to a skew-morphism of $\Gamma$, $\pi(\omega)=t$ for all $\omega\in\Omega$ and $\pi(\eta)\in\{1,t\}$ for all $\eta\in\Gamma$.

{\rm(c)} When the conditions in {\rm(b)} are satisfied,
$\Gamma_+:=\{\eta\in \Gamma\colon \pi(\eta)=1\}$ is a subgroup of index $2$,
consisting of elements which are products of an even number of generators,
$\varphi(\Gamma_+)=\Gamma_+$, and $\varphi_+:=\varphi|_{\Gamma_+}$ is an automorphism.
\end{prop}

By (\ref{eq:t-balance}), there is an involution $\iota$ on $\{1,\ldots,d\}$ with $\omega_i^{-1}=\omega_{\iota(i)}$ and $\iota(i+1)\equiv\iota(i)+t\pmod{d}$ for all $i$.
Let $\ell=\iota(d)$, then $\iota(i)\equiv\ell+ti\pmod{d}$, and the condition $\iota^2={\rm id}$ is equivalent to $(t+1)\ell\equiv 0\pmod{d}$, which together with $t^2\equiv 1\pmod{d}$ implies $(t-1,d)\mid 2\ell$.
We say that the RBCM$_t$ has {\it type I} or {\it type II} if $(t-1,d)\nmid\ell$ or $(t-1,d)\mid\ell$, respectively.

\begin{rmk} \label{rmk:ell-iso}
\rm Observe that $(t-1,d)\mid\ell$ if and only if $\Omega$ contains an element of order 2, so RBCM$_t$'s of different type cannot be isomorphic. On the other hand, according to Lemma 2.4 of \cite{KKF06}, two RBCM$_t$'s of the same type $\mathcal{CM}(\Gamma_j,\Omega_j,\rho_j), j=1,2$ are isomorphic if and only if there exists an isomorphism $\sigma:\Gamma_1\to\Gamma_2$ such that $\sigma(\Omega_1)=\Omega_2$ and $\sigma\circ\rho_1=\rho_2\circ\sigma$.

When $\mathcal{CM}(\Gamma,\{\omega_1,\ldots,\omega_d\},\rho)$ has type I (resp. type II), by re-indexing the $\omega_i$'s if necessary, we may assume $\ell=(t-1,d)/2$ (resp. $\ell=(t-1,d)$).
\end{rmk}

So far, people have completely classified RBCM$_{t}$'s for the following classes of groups: dihedral groups (Kwak, Kwon and Feng \cite{KKF06}, 2006), dicyclic groups (Kwak and Oh \cite{KO08}, 2008), semi-dihedral groups (Oh \cite{Oh09}, 2009), cyclic groups (Kwon \cite{Kw13}, 2013).
In 2017 the first author \cite{Ch17} reduced the classification of RBCM$_t$'s on abelian groups to a problem about polynomial rings, and gave a complete classification for RBCM$_t$'s on abelian $2$-groups. In 2018 Yuan, Wang and Qu \cite{YWQ18} classified RBCM$_1$'s for the so-called minimal nonabelian metacyclic groups. For results on more general regular Cayley maps, see \cite{CDL22,CT14,DH19,DYL23,HR22}.

It is still challenging to study regular Cayley maps on nonabelian groups.
We propose a ``reduction method", through which known results about RBCM$_t$'s on simpler groups may be applicable.
A key ingredient is the following observation.
\begin{lem}  \label{lem:quotient}
Let $\mathcal{CM}(\Gamma,\Omega,\rho)$ be a RBCM$_{t}$ with skew-morphism $\varphi$.
Suppose $\Xi$ is a normal subgroup of $\Gamma$ which is contained in $\Gamma_+$ and invariant under $\varphi_+$.
Let $\overline{\Gamma}=\Gamma/\Xi$, and let $\overline{\Omega}$ denote the image of $\Omega$ under the quotient map $\Gamma\twoheadrightarrow\overline{\Gamma}$. Then $\rho$ induces a permutation $\overline{\rho}$ on $\overline{\Omega}$ and gives rise to a RBCM$_t$ $\mathcal{CM}(\overline{\Gamma},\overline{\Omega},\overline{\rho})$.
Furthermore, if $\mathcal{CM}(\Gamma,\Omega,\rho)$ has type II, then so does $\mathcal{CM}(\overline{\Gamma},\overline{\Omega},\overline{\rho})$.
\end{lem}

\begin{proof}
For $\eta\in\Gamma$, let $\overline{\eta}$ denote its image under the quotient map $\Gamma\twoheadrightarrow\overline{\Gamma}$.

The map $\overline{\varphi}:\overline{\Gamma}\to\overline{\Gamma}$, $\overline{\eta}\mapsto \overline{\varphi(\eta)}$
is well-defined, as $\varphi(\xi\eta)=\varphi(\xi)\varphi(\eta)$ for any $\xi\in\Xi$. Let $\pi$ be the power function of $\mathcal{CM}(\Gamma,\Omega,\rho)$. It induces a function $\overline{\pi}:\overline{\Gamma}\twoheadrightarrow\{1,t\}$ in an obvious way.
For all $\eta,\mu$, we have
$$\overline{\varphi}(\overline{\eta}\overline{\mu})=\overline{\varphi(\eta\mu)}=\overline{\varphi(\eta)\varphi^{\pi(\eta)}(\mu)}=
\overline{\varphi}(\overline{\eta})\overline{\varphi}^{\overline{\pi}(\overline{\eta})}(\overline{\mu}).$$
So $\rho=\varphi|_\Omega$ induces a permutation $\overline{\rho}$ on $\overline{\Omega}$, building $\mathcal{CM}(\overline{\Gamma},\overline{\Omega},\overline{\rho})$ into a RBCM$_t$.

The assertion about type follows from the first sentence of Remark \ref{rmk:ell-iso}.
\end{proof}

The idea is, to understand a RBCM$_t$ $\mathcal{M}$ on $\Gamma$, we take a suitable subgroup $\Xi$, investigate the quotient RBCM$_t$ $\overline{\mathcal{M}}$ on $\Gamma/\Xi$, and use knowledge on $\overline{\mathcal{M}}$ to extract information about $\mathcal{M}$ as much as possible.

In this paper, we apply the reduction method to classify RBCM$_t$'s for a class of split metacyclic 2-groups.

A general {\it split metacyclic group} can be presented as
\begin{align}
\Lambda(n,m;r)=\langle \alpha,\beta\mid \alpha^{n}=\beta^{m}=1, \ \beta\alpha\beta^{-1}=\alpha^{r}\rangle, \label{eq:presentation}
\end{align}
for some positive integers $n,m,r$ such that $r^m\equiv 1\pmod{n}$; see \cite{GG09} Page 2.
We focus on $\Lambda(2^a,2^b;1+2^c)$, with
\begin{align}
\max\{2,a-b\}\le c\le a-3  \qquad \text{and} \qquad b\ne c.  \label{eq:restriction}
\end{align}
These groups constitute a major part of split metacyclic $2$-groups of Class A, as introduced on \cite{Cu07} Page 2. The artificial restriction (\ref{eq:restriction}) is imposed for simplicity, so that the paper has a clear structure and a moderate length;
if $b=c$ is allowed, then some annoying subtleties will arise, but nothing interesting will happen.

The main result is Theorem \ref{thm:main}.
As shown in \cite{YWQ18}, any metacyclic $p$-group for odd prime $p$ does not admit a RBCM$_1$; (by Proposition \ref{prop:RBCMt}, it does not admit a RBCM$_t$ for $t>1$). On the contrary, we shall see that the metacyclic $2$-group $\Lambda(2^a,2^b;1+2^c)$ admits a rich family of RBCM$_t$'s, consisting of $2^{a-c-1}$ isomorphism classes. To some extent, we can say that the richness and complexity of RBCM$_t$'s on metacyclic groups are concentrated on metacyclic $2$-groups.

Section 2 presents a preliminary on metacyclic groups. Section 3 comprises the main steps of classifying RBCM$_t$'s.
First, we combine Lemma \ref{lem:quotient} and the previous work \cite{Ch17} on RBCM$_t$'s on abelian $2$-groups to deduce several constraints on RBCM$_t$'s on metacyclic $2$-groups, stated as Lemma \ref{lem:constraint}. Second, based on the work \cite{CXZ18} on automorphisms of metacyclic groups,  we show that each RBCM$_t$ can be ``normalized", in the sense that it is isomorphic to one with the property that $\varphi_+$ and $\omega_d$ are in certain special forms. Third, we solve a system of congruence equations which characterize conditions for given data to determine a normalized RBCM$_t$. Finally we state the classification as Theorem \ref{thm:main}.

\begin{nota}
\rm For positive integers $u,s$, let $[u]_{s}=1+s+\cdots+s^{u-1}$; let $[0]_s=0$.

For $u\ne 0$, let $\|u\|$ denote the largest $k$ with $2^k\mid u$; set $\|0\|=+\infty$.

For an element $\theta$ of a finite group, let $|\theta|$ denote its order.

Let $\mathbb{Z}_n=\mathbb{Z}/n\mathbb{Z}$, which is a quotient ring of $\mathbb{Z}$.

For an abelian $2$-group $\Gamma$, let ${\rm rk}(\Gamma)$ denote its rank.

Given a normal subgroup $\Xi\vartriangleleft\Gamma$, the image of $\eta\in\Gamma$ under the quotient $\Gamma\twoheadrightarrow\Gamma/\Xi$ is usually denoted by $\overline{\eta}$, (but for $u\in\mathbb{Z}$, its image under $\mathbb{Z}\twoheadrightarrow\mathbb{Z}_n$ is still denoted by $u$), and if an automorphism $\phi$ of $\Gamma$ satisfies $\phi(\Xi)=\Xi$, then its induced automorphism on $\Gamma/\Xi$ is denoted by $\overline{\phi}$.

A RBCM$_t$ $\mathcal{CM}(\Gamma,\Omega,\rho)$ is shorten as $\mathcal{CM}(\Gamma,\Omega)$ if $\Omega$ can be written as $\{\omega_1,\ldots,\omega_d\}$ and $\rho(\omega_i)=\omega_{i+1}$. The subscript in $\omega_i$ is always understood as modulo $d$.
Let ${\rm Aut}^+(\Gamma)=\{\tau\in{\rm Aut}(\Gamma)\colon \tau(\Gamma_+)=\Gamma_+\}$.

Since various congruences modulo powers of $2$ will appear in the computations, to simplify the writing we use $A\equiv^{(k)}B$ to indicate $A\equiv B\pmod{2^k}$.
Furthermore, abbreviate $A\equiv^{(a-1)}B$ to $A\equiv B$, and $A\equiv^{(b)}B$ to $A\equiv'B$.
\end{nota}

\section{Preliminary on metacyclic groups}

A general element of $\Lambda=\Lambda(n,m;r)$ can be written as $\alpha^{x}\beta^{y}$.
By (\ref{eq:presentation}) we have
\begin{align}
\beta^{y}\alpha^{x}&=\alpha^{xr^{y}}\beta^{y}, \nonumber  \\
(\alpha^{x_{1}}\beta^{y_{1}})(\alpha^{x_{2}}\beta^{y_{2}})&=\alpha^{x_{1}+x_{2}r^{y_{1}}}\beta^{y_{1}+y_{2}},  \nonumber   \\
(\alpha^{x}\beta^{y})^{u}&=\alpha^{x[u]_{r^{y}}}\beta^{yu},  \label{eq:power} \\
[\alpha^{x_{1}}\beta^{y_{1}},\alpha^{x_{2}}\beta^{y_{2}}]&=\alpha^{x_{1}(1-r^{y_{2}})-x_{2}(1-r^{y_{1}})}. \nonumber
\end{align}
Here $r^y$ is understood as $r^{y-m[y/m]}$ if $y<0$, $[u]_{r^y}$ is understood as $[u-n[u/n]]_{r^y}$ if $u<0$, and the commutator $[\eta,\mu]=\eta\mu\eta^{-1}\mu^{-1}$. Consequently, the commutator subgroup is generated by $\langle\alpha^{r-1}\rangle$, hence the abelianization
\begin{align*}
\Lambda^{{\rm ab}}:=\Lambda/[\Lambda,\Lambda]\cong\mathbb{Z}_{(r-1,n)}\times\mathbb{Z}_{m}.
\end{align*}

\begin{lem}  \label{lem:index-2}
There are three index $2$ subgroups of $\Lambda=\Lambda(n,m;r)$, namely, $\langle \alpha^{2},\beta\rangle$, $\langle \alpha,\beta^{2}\rangle$ and $\langle \alpha^{2},\alpha\beta\rangle$.
\end{lem}
\begin{proof}
Each homomorphism $\Lambda\to\mathbb{Z}_2$ factors through $\Lambda^{{\rm ab}}$, and there are exactly three epimorphisms $\kappa_j:\Lambda^{{\rm ab}}\cong\mathbb{Z}_{(r-1,n)}\times\mathbb{Z}_{m}\twoheadrightarrow\mathbb{Z}_2, j=1,2,3$, given by
\begin{align*}
\kappa_1(u,v)=u, \qquad \kappa_2(u,v)=v, \qquad \kappa_3(u,v)=u+v.
\end{align*}
Let $\widetilde{\kappa}_j$ denote the composite of the quotient $\Lambda\twoheadrightarrow\Lambda^{{\rm ab}}$ and $\kappa_j$. It is easy to see that $\ker\widetilde{\kappa}_1=\langle\alpha^2,\beta\rangle$, $\ker\widetilde{\kappa}_2=\langle\alpha,\beta^2\rangle$, $\ker\widetilde{\kappa}_3=\langle\alpha^2,\alpha\beta\rangle$.
\end{proof}

The following is a special case of \cite{CXZ18} Theorem 2.9, in which, $\Lambda_1=\{2\}$, $\Lambda_2=\Lambda'=\emptyset$, $a_2=c_2=a$, $b_2=b$, $d_2=c$, $t=2^a$, $m=2^b$, $m_0=1$.

If there is an automorphism $\sigma$ of $\Lambda(n,m;r)$ sending $\alpha$ and $\beta$ to $\alpha^{x_1}\beta^{y_1}$ and $\alpha^{x_2}\beta^{y_2}$, respectively, then we denote such automorphism $\sigma$ by $\sigma^{x_1,y_1}_{x_{2},y_{2}}$ in this paper.
\begin{lem} \label{lem:auto}
Suppose $\|r-1\|=c\ge 2$. Each automorphism of $\Lambda(2^a,2^b;r)$ is given by
$\sigma^{x_1,y_1}_{x_{2},y_{2}}:\alpha\mapsto\alpha^{x_1}\beta^{y_1}, \beta\mapsto\alpha^{x_2}\beta^{y_2}$
for some integers $x_{1},y_{1},x_{2},y_{2}$ with
\begin{align*}
2&\nmid x_{1}y_{2}-x_{2}y_{1}, \qquad \|y_1\|\ge b-c, \qquad \|x_2\|\ge a-b, \\
y_2&\equiv^{(a-c)}\begin{cases}  1+2^{a-c-1}, &\text{if}\ b=a-c=\|y_{1}\|+c, \\  1,&\text{otherwise}. \end{cases}
\end{align*}
Actually, any $x_{1},y_{1},x_{2},y_{2}$ satisfying these define an automorphism.

Acting on general elements,
\begin{align*}
\sigma^{x_1,y_1}_{x_{2},y_{2}}(\alpha^{u}\beta^{v})=\alpha^{x_{1}[u]_{r^{y_{1}}}+r^{y_{1}u}x_{2}[v]_{r^{y_{2}}}}\beta^{y_{1}u+y_{2}v}.
\end{align*}
\end{lem}

Given $\sigma^{x_1,y_1}_{x_2,y_2}$ and $\sigma^{p_1,q_1}_{p_2,q_2}$, the composite $\sigma^{p_1,q_1}_{p_2,q_2}\circ\sigma^{x_1,y_1}_{x_2,y_2}$ sends $\alpha$ to $\alpha^{h_1}\beta^{q_1x_1+q_2y_1}$ and sends $\beta$ to $\alpha^{h_2}\beta^{q_1x_2+q_2y_2}$, with
\begin{align*}
h_j=p_1[x_j]_{r^{q_1}}+r^{q_1x_j}p_2[y_j]_{r^{q_2}}, \qquad j=1,2.
\end{align*}

Let $r=1+2^c$. Since $2(\|q_1\|+c)\ge b+c\ge a$, we have $r^{q_1u}\equiv^{(a)}1+2^cq_1u$, so
\begin{align*}
[x_j]_{r^{q_1}}&={\sum}_{i=0}^{x_j-1}r^{iq_1}\equiv^{(a)}x_j+2^{c-1}q_1x_j(x_j-1), \\
r^{q_1x_j}p_2&\equiv^{(a)}p_2+2^cq_1x_jp_2\equiv^{(a)}p_2.
\end{align*}

Suppose $c>b$ which will hold in the next section. Then $\|x_2\|>a-c$, $\|p_2\|>a-c$, so that $2^{c-1}x_2\equiv^{(a)}0$, and $p_2[y_j]_{r^{q_2}}\equiv^{(c)}p_2y_j$, implying
$$h_1\equiv^{(a)}p_1x_1(1+2^{c-1}q_1(x_1-1))+p_2y_1, \qquad  h_2\equiv^{(a)}p_1x_2+p_2y_2.$$
Thus
\begin{align}
\sigma^{p_1,q_1}_{p_2,q_2}\circ\sigma^{x_1,y_1}_{x_2,y_2}=\sigma^{p_1x_1(1+2^{c-1}q_1(x_1-1))+p_2y_1,q_1x_1+q_2y_1}_{p_1x_2+p_2y_2,q_1x_2+q_2y_2}. \label{eq:composite}
\end{align}

\section{Classifying regular $t$-balanced Cayley maps for a class of split metacyclic 2-groups}

Let $\Delta=\Lambda(2^a,2^b;1+2^c)$ for $(a,b,c)$ satisfying (\ref{eq:restriction}). In particular, $b\ge 3$, $c\ge 2$.

By \cite{CXZ18} Lemma 2.1, $\|[u]_{(1+2^c)^{y}}\|=\|u\|$. Then by (\ref{eq:power}),
\begin{align}
|\alpha^{x}\beta^{y}|=2^{\max\{a-\|x\|, b-\|y\|\}}.   \label{eq:order}
\end{align}

Suppose $\mathcal{CM}(\Delta,\{\omega_1,\ldots,\omega_d\})$ is a RBCM$_t$ with skew-morphism $\varphi$.
As in Remark \ref{rmk:ell-iso}, we may further assume $\ell\in\{(t-1,d)/2,(t-1,d)\}$, so
\begin{align}
\omega_{\ell+ti}=\omega_i^{-1}, \qquad i=1,\ldots,d. \label{eq:inverse}
\end{align}
Let
$\eta_j=\omega_j\omega_{j-1}^{-1}=\omega_j\omega_{\ell+t(j-1)}$.
Then
\begin{align}
\Delta_+&=\langle\eta_1,\ldots,\eta_d\rangle,  \label{eq:Delta+}  \\
\omega_i\omega_d^{-1}&=\eta_i\cdots\eta_1,  \qquad  i=1,\ldots,d;  \nonumber
\end{align}
in particular,
\begin{align}
\omega_d^{-2}=\omega_{\ell}\omega_d^{-1}=\eta_\ell\cdots\eta_1.  \label{eq:eta-ell-1}
\end{align}
Moreover, $\varphi(\omega_j\omega_{\ell+t(j-1)})=\omega_{j+1}\varphi^{t}(\omega_{\ell+t(j-1)})=\omega_{j+1}\omega_{\ell+tj}$, i.e.,
\begin{align}
\varphi_+(\eta_j)=\eta_{j+1}.    \label{eq:phi-eta}
\end{align}

\subsection{Constraints}

\begin{lem}  \label{lem:RBCMt-ab}
Suppose $\mathcal{CM}(\Gamma,\{\mu_1,\ldots,\mu_m\})$ is a RBCM$_t$ with skew-morphism $\psi$, and $\Gamma$ is an abelian 2-group such that ${\rm rk}(\Gamma)={\rm rk}(\Gamma_+)=2$. Then
\begin{enumerate}
  \item[\rm(i)] there exists an isomorphism $\Gamma_+\cong\mathbb{Z}_{2^{k'}}\times\mathbb{Z}_{2^{k}}$ for some $k'\ge k\ge1$, sending
       $\theta_1$ to $(1,0)$ and $\theta_2$ to $(-1,1)$, where $\theta_j=\mu_j-\mu_{j-1}$;
  \item[\rm(ii)] $\mathcal{CM}(\Gamma,\{\mu_1,\ldots,\mu_m\})$ has type I, and $m=2^{k+1}\mid t+1$;
  \item[\rm(iii)] $\psi_+^2={\rm id}$.
\end{enumerate}
\end{lem}

\begin{proof}
RBCM$_t$'s on abelian $2$-groups were completely classified in \cite{Ch17} Section 4.2, Corollary 4.3 and Corollary 4.7; obviously ${\rm rk}(\Gamma)={\rm rk}(\Gamma_+)=2$ only occurs in the last case of Section 4.2. The conditions (i)--(iii) can be easily verified.
\end{proof}

\begin{lem} \label{lem:constraint}
For our RBCM$_t$ $\mathcal{CM}(\Delta,\{\omega_1,\ldots,\omega_d\})$, the following holds:
\begin{enumerate}
  \item[\rm(i)] it has type I, with $\Delta_+=\langle\alpha^2,\beta\rangle\cong\Lambda(2^{a-1},2^b;1+2^c)$;
  \item[\rm(ii)] $c>b$, and $\|t+1\|>b$;
  \item[\rm(iii)] $\varphi_+=\sigma^{x_1,y_1}_{x_2,y_2}$ for some $x_1,y_1,x_2,y_2$ with $2\nmid y_1$ and
        $$x_1^2+x_2y_1\equiv^{(c-1)}1, \qquad  x_1+y_2\equiv'y_2^2+x_2y_1-1\equiv'0.$$
\end{enumerate}
\end{lem}

\begin{rmk}\label{rmk:care}
\rm As a consequence of (ii), $c\ge 4$.

Be careful: here $\varphi_+=\sigma^{x_1,y_1}_{x_2,y_2}$ means that it sends $\alpha^2$ to $\alpha^{2x_1}\beta^{y_1}$ and sends $\beta$ to $\alpha^{2x_2}\beta^{y_2}$.
\end{rmk}

\begin{proof}
The proof consists of three parts.
\begin{enumerate}
  \item Assume $\Delta_+=\langle\alpha^2,\alpha\beta\rangle$, $\eta_j=\alpha^{u_j}\beta^{v_j}$ ($j=1,\ldots,d$), and
        $$\varphi_+(\alpha^2)=(\alpha^2)^{x_1}(\alpha\beta)^{y_1}, \qquad \varphi_+(\alpha\beta)=(\alpha^2)^{x_2}(\alpha\beta)^{y_2}.$$
        Since $|(\alpha^2)^{x_2}(\alpha\beta)^{y_2}|=|\varphi_+(\alpha\beta)|=|\alpha\beta|$, by (\ref{eq:order}) we have $2\nmid y_2$;
        from
        $$1=\varphi_+\big((\alpha^2)^{2^{a-1}}\big)=((\alpha^2)^{x_1}(\alpha\beta)^{y_1})^{2^{a-1}}=(\alpha^{2x_1+[y_1]_{1+2^c}}\beta^{y_1})^{2^{a-1}}$$
        we see $2\mid y_1$. From (\ref{eq:phi-eta}) we see that all the $v_j$'s have the same parity, which, by (\ref{eq:Delta+}), must be odd.
        By (\ref{eq:eta-ell-1}), $2\mid v_1+\cdots+v_\ell$, so $\ell$ is even.

        On the other hand, one can verify that the subgroup
        $$\Xi=\big\langle\alpha^{2^c},\beta^{2^{c}}\big\rangle=\big\langle\alpha^{2^{c}},(\alpha\beta)^{2^c}\big\rangle$$
        is normal in $\Delta$ and invariant under $\varphi_+$. By Lemma \ref{lem:quotient} there is a quotient RBCM$_t$ $\overline{\mathcal{M}}$ on $\Delta/\Xi\cong\mathbb{Z}_{2^c}\times\mathbb{Z}_{2^c}$.
        Clearly ${\rm rk}(\Delta/\Xi)={\rm rk}(\Delta_+/\Xi)=2$, hence by Lemma \ref{lem:RBCMt-ab}, $\overline{\mathcal{M}}$ has type I, and
        $4\mid t+1$. By Lemma \ref{lem:quotient}, our RBCM$_t$ has type I. Hence $\ell=(t-1,d)/2$, contradicting $2\mid\ell$.
  \item Assume $\Delta_+=\langle\alpha,\beta^2\rangle$, $\eta_j=\alpha^{u_j}\beta^{2v_j}$ ($j=1,\ldots,d$)
        and $\varphi_+=\sigma^{x_1,y_1}_{x_2,y_2}$.
        Since $\Delta_+\cong\Lambda(2^a,2^{b-1};(1+2^c)^2)$, by Lemma \ref{lem:auto} we have
        \begin{align}
        \|y_1\|\ge b-c-2,  \quad \|x_2\|\ge a-b+1, \quad   \|y_2-1\|\ge a-c-2.  \label{ineq:deg-x2-y1-y2}
        \end{align}
        The subgroup $\Xi'=\big\langle\alpha^{2^c},\beta^{2^{b-1}}\big\rangle$ is normal in $\Delta$ and invariant under $\varphi_+$, with $\Delta/\Xi'\cong\mathbb{Z}_{2^c}\times\mathbb{Z}_{2^{b-1}}$ and $\Delta_+/\Xi'\cong\mathbb{Z}_{2^c}\times\mathbb{Z}_{2^{b-2}}$.
        In $\Delta_+/\Xi'$,
        \begin{align*}
        \overline{\eta_1}&=u_1\overline{\alpha}+v_1\overline{\beta^2},   \\
        \overline{\eta_1}+\overline{\eta_2}&=((x_1+1)u_1+x_2v_1)\overline{\alpha}+(y_1u_1+(y_2+1)v_1)\overline{\beta^2}.
        \end{align*}
        By Lemma \ref{lem:quotient} and Lemma \ref{lem:RBCMt-ab}, $4\mid t+1$ and $\ell=(t-1,d)/2$, so that $2\nmid\ell$.

        If $2\mid x_2$, then $2\nmid x_1$, so that $u_j\equiv u_1\pmod{2}$; by (\ref{eq:Delta+}), $2\nmid u_1$, hence $\eta_{\ell}\cdots\eta_1=\alpha^{u'}\beta^{v'}$ for some odd $u'$, but this contradicts (\ref{eq:eta-ell-1}).
        Hence $2\nmid x_2$, and consequently
        $b-1\ge a\ge c+3$.

        By Lemma \ref{lem:RBCMt-ab} (i), $|\overline{\eta_1}|=2^{b-2}$ and $|\overline{\eta_1}+\overline{\eta_2}|=2^c$. Hence $2\nmid v_1$, and
        $$c\ge b-2-\|y_1u_1+(y_2+1)v_1\|=b-3;$$
        the inequality comes from (\ref{eq:order}), and the equality relies on (\ref{ineq:deg-x2-y1-y2}) which implies $\|y_1\|\ge b-c-2\ge 2>\|y_2+1\|=1$. This contradicts $b-1\ge c+3$.
  \item Therefore by Lemma \ref{lem:index-2},
        $\Delta_+=\langle\alpha^2,\beta\rangle\cong\Lambda(2^{a-1},2^b;1+2^c)$. Suppose $\varphi_+=\sigma^{x_1,y_1}_{x_2,y_2}$,
        and $\eta_j=\alpha^{2u_j}\beta^{v_j}$.

        The subgroup $\langle\alpha^{2^c}\rangle$ is normal in $\Delta$ and invariant under $\varphi_+$. Applying Lemma \ref{lem:RBCMt-ab} to the quotient RBCM$_t$ on $\Delta/\langle\alpha^{2^c}\rangle\cong\mathbb{Z}_{2^c}\times\mathbb{Z}_{2^b}$, we obtain $4\mid t+1$ and $\ell=(t-1,d)/2$. So $2\nmid\ell$.

        If $2\mid y_1$, then $v_j\equiv v_1\pmod{2}$ for all $j$;
        by (\ref{eq:Delta+}), $2\nmid v_1$, and since $2\nmid \ell$, we have
        $\eta_{\ell}\cdots\eta_1=\alpha^{2u'}\beta^{v'}$ for some odd $v'$, contradicting (\ref{eq:eta-ell-1}).
        Hence $2\nmid y_1$, and then by Lemma \ref{lem:auto}, $b\le c+\|y_1\|=c$.
        Since $b\ne c$, actually $c>b$.

        By Lemma \ref{lem:RBCMt-ab} (iii), $\overline{\varphi_+}^2={\rm id}$.
        The expression for $\overline{\varphi_+}^2$ is
        \begin{align*}
        \overline{\alpha^2}&\mapsto (x_1^2+x_2y_1)\overline{\alpha^2}+y_1(x_1+y_2)\overline{\beta}, \\
        \overline{\beta}&\mapsto x_2(x_1+y_2)\overline{\alpha^2}+(x_2y_1+y_2^2)\overline{\beta}.
        \end{align*}
        Thus, $x_1^2+x_2y_1\equiv^{(c-1)}1$ and $x_1+y_2\equiv'y_2^2+x_2y_1-1\equiv'0$.
\end{enumerate}
\end{proof}

\subsection{Normalization}

\begin{lem}\label{lem:restriction}
Let $\sigma^{z_1,w_1}_{z_2,w_2}\in{\rm Aut}(\Delta_+)$. There exists $\tau\in{\rm Aut}^+(\Delta)$ with $\tau_+=\sigma^{z_1,w_1}_{z_2,w_2}$ if and only if $2\mid w_1$ and $\|w_2-1\|\ge a-c$.
\end{lem}

\begin{proof}
If $\tau=\sigma^{p_1,q_1}_{p_2,q_2}\in{\rm Aut}^+(\Delta)$, then by Lemma \ref{lem:auto}, $2\mid p_2$ and $\|q_2-1\|\ge a-c$.
As an automorphism of $\Delta_+$, $\tau_+(\alpha^2)=\alpha^{p_1(2+2^cq_1)}\beta^{2q_1}$, $\tau_+(\beta)=\alpha^{p_2}\beta^{q_2}$,
hence
\begin{align*}
\tau_+=\sigma^{p_1(1+2^{c-1}q_1),2q_1}_{p_2/2,q_2}.
\end{align*}
So $2\mid w_1$ and $\|w_2-1\|\ge a-c$ are necessary for there to exist $\tau\in{\rm Aut}^+(\Delta)$ with $\tau_+=\sigma^{z_1,w_1}_{z_2,w_2}$.

Conversely, suppose $2\mid w_1$ and $\|w_2-1\|\ge a-c$. Put
$$\tau=\sigma^{(1-2^{c-2}w_1)z_1,w_1/2}_{2z_2,w_2}.$$
It is clear that $\tau\in{\rm Aut}^+(\Delta)$ with $\tau_+=\sigma^{z_1,w_1}_{z_2,w_2}$.
\end{proof}

\begin{lem}  \label{lem:quadratic}
Suppose $h$ is odd, $e>2$, and $s^2\equiv^{(e)}h$. There exists a sequence $\{\tilde{s}_k\}_{k=2}^{\infty}$ such that $\tilde{s}_k^2\equiv^{(k(e-1))}h$ and $\tilde{s}_{k+1}\equiv^{(k(e-1)-1)}\tilde{s}_k$ for each $k$.
Consequently, for each $\tilde{e}>e$, there exists $\tilde{s}$ such that $\tilde{s}^2\equiv^{(\tilde{e})}h$ and $\tilde{s}\equiv^{(e-1)}s$.
\end{lem}

\begin{proof}
We construct $\tilde{s}_k$ recursively.
Since $s$ is odd, we may take $a_1\in\mathbb{Z}$ such that $h\equiv^{(2(e-1))}s^2+2^esa_1$. Set $\tilde{s}_2=s+2^{e-1}a_1$. Then clearly $\tilde{s}_2^2\equiv^{(2(e-1))}h$ and $\tilde{s}_2\equiv^{(e-1)}s$.

Assume $k\ge 2$ and $\tilde{s}_k$ has been obtained. Take $a_k\in\mathbb{Z}$ with $$h\equiv^{((k+1)(e-1))}\tilde{s}_k^2+2^{k(e-1)}\tilde{s}_ka_k,$$
and set
$$\tilde{s}_{k+1}=\tilde{s}_k+2^{k(e-1)-1}a_k.$$
Then $\tilde{s}_{k+1}\equiv^{(k(e-1)-1)}\tilde{s}_k$ and $\tilde{s}_{k+1}^2\equiv^{((k+1)(e-1))}h$, due to $2k(e-1)-2\ge (k+1)(e-1)$.
\end{proof}

\begin{lem}  \label{lem:normal-auto}
There exists $\tau_1\in{\rm Aut}^+(\Delta)$ such that $(\tau_1\varphi\tau_1^{-1})_+=\sigma^{z,1}_{0,-z}$ for some $z\equiv^{(c-2)}-1$.
\end{lem}

\begin{proof}
We are going to find $u_1,v_1,u_2,v_2, z,w$ satisfying the following:
\begin{align}
u_1x_1(1+2^{c-1}v_1(x_1-1))+u_2y_1&\equiv zu_1(1+2^{c-1}(u_1-1)), \label{eq:I2-x-1} \\
v_1x_1+v_2y_1&\equiv'u_1+wv_1, \label{eq:I2-y-1} \\
u_1x_2+u_2y_2&\equiv zu_2, \label{eq:I2-x-2}  \\
v_1x_2+v_2y_2&\equiv'u_2+wv_2. \label{eq:I2-y-2}
\end{align}
In view of (\ref{eq:composite}), these will ensure
$$\sigma^{u_1,v_1}_{u_2,v_2}\circ\varphi_+=\sigma^{z,1}_{0,w}\circ\sigma^{u_1,v_1}_{u_2,v_2}.$$

Take $\gamma$ with $(1+2^{c-1}(y_1-1))\gamma\equiv 1$, and let
$$f(x)=(x-\gamma x_1)(x-y_2)-x_2y_1=\Big(x-\frac{\gamma x_1+y_2}{2}\Big)^2-x_2y_1-\Big(\frac{\gamma x_1+y_2}{2}\Big)^2.$$
Remember that $c\ge a-b>a-c\ge 3$ and $\|x_2\|\ge a-b$. Then $f(x_1)\equiv^{(a-b)}0$. By Lemma \ref{lem:quadratic}, there exists $z$ with
$f(z)\equiv0$ and $z\equiv^{(a-b-1)}x_1$. Note that $\|z-y_2\|=\|x_1-y_2\|=1$.

Let $u_1=y_1$, $v_1=0$, $v_2=1$, $w=y_2-u_2$, and $u_2=(1+2^{c-1}(y_1-1))z-x_1.$
Then $f(z)\equiv0$ is equivalent to
$$(z-y_2)u_2\equiv x_2y_1.$$
It is easy to verify that (\ref{eq:I2-x-1})--(\ref{eq:I2-y-2}) all hold.
Now it holds that
$$\|u_2\|=\|x_2\|+\|y_1\|-\|z-y_2\|=\|x_2\|-1\ge a-c-1,$$
by Lemma \ref{lem:restriction}, $\sigma^{u_1,v_1}_{u_2,v_2}=(\tau_1)_+$ for some $\tau_1\in{\rm Aut}^+(\Delta)$.

Consider the automorphism of $\Delta_+/\langle\alpha^{2^c}\rangle$ induced by $\tau_1\varphi_+\tau_1^{-1}$. Similarly as the final part of the proof of Lemma \ref{lem:constraint}, we have $z^2\equiv^{(c-1)}0$ and $z+w\equiv'w^2-1\equiv'0$.
Thus $(\tau_1\varphi\tau_1^{-1})_+=\sigma^{z,1}_{0,-z}$. Note that $w\equiv y_2\equiv 1\pmod{4}$, implying $\|z+1\|\ge c-2$.
\end{proof}

\begin{lem} \label{lem:conjugate}
Suppose $x\equiv^{(c-2)}z\equiv^{(c-2)}-1$.
If $\tau\in{\rm Aut}^+(\Delta)$ with $\tau_+=\sigma^{p_1,q_1}_{p_2,q_2}$, then
$\tau_+\circ\sigma^{z,1}_{0,-z}\circ\tau_+^{-1}=\sigma^{x,1}_{0,-x}$ is equivalent to
\begin{align}
\|p_2\|\ge a-2, \qquad    p_1-q_2\equiv'2zq_1,  \qquad x\equiv z+p_2. \label{eq:conjugate}
\end{align}
In particular, $\tau_+\circ\sigma^{z,1}_{0,-z}\circ\tau_+^{-1}=\sigma^{z,1}_{0,-z}$ if and only if $p_2\equiv 0$ and $p_1-q_2\equiv'2zq_1.$
\end{lem}

\begin{proof}
By Lemma \ref{lem:restriction},  $2\mid q_1$ and $\|q_2-1\|\ge a-c$.

By (\ref{eq:composite}), $\tau_+\circ\sigma^{z,1}_{0,-z}=\sigma^{x,1}_{0,-x}\circ\tau_+$ is equivalent to
\begin{align}
p_1z\big(1+2^{c-1}q_1(z-1)\big)+p_2&\equiv xp_1\big(1+2^{c-1}(p_1-1)\big), \label{eq:conj-1} \\
q_1z+q_2&\equiv'p_1-xq_1, \label{eq:conj-2} \\
-p_2z&\equiv xp_2, \label{eq:conj-3} \\
-q_2z&\equiv'p_2-xq_2. \label{eq:conj-4}
\end{align}

Since $x\equiv^{(c-2)}z\equiv^{(c-2)}-1$, we have $\|x+z\|=1$, hence by (\ref{eq:conj-3}), $\|p_2\|\ge a-2$. Then (\ref{eq:conj-4}) implies $x\equiv^{(a-2)}z$, and consequently by (\ref{eq:conj-2}), $p_1-q_2\equiv'2zq_1$. Now it holds that $c-1\ge b\ge a-c$, and one has
$$p_1-1=(p_1-q_2)+(q_2-1)\equiv^{(a-c)}2zq_1\equiv^{(a-c)}(z-1)q_1,$$
which together with (\ref{eq:conj-1}) implies
\begin{align*}
zp_1\big(1+2^{c-1}(z-1)q_1\big)+p_2\equiv xp_1\big(1+2^{c-1}(z-1)q_1\big).
\end{align*}
Since $p_2\equiv p_2\cdot p_1\big(1+2^{c-1}(z-1)q_1\big)$, we have $x\equiv z+p_2$.

Conversely, assuming (\ref{eq:conjugate}), it is rather easy to deduce (\ref{eq:conj-1})--(\ref{eq:conj-4}).
\end{proof}

\begin{lem} \label{lem:normal-u}
{\rm(i)} There exists $\tau_2\in{\rm Aut}^+(\Delta)$ such that $(\tau_2)_+\circ\sigma^{z,1}_{0,-z}=\sigma^{z,1}_{0,-z}\circ(\tau_2)_+$ and $(\tau_2\tau_1)(\omega_d)=\alpha^{\tilde{u}}\beta$ for some odd $\tilde{u}$.

{\rm(ii)} For any $\tilde{u}'$ with $\tilde{u}'\equiv^{(a-c)}\tilde{u}$, there exists $\tau\in{\rm Aut}^+(\Delta)$ such that $\tau_+\circ\sigma^{z,1}_{0,-z}=\sigma^{z,1}_{0,-z}\circ\tau_+$ and $\tau(\alpha^{\tilde{u}}\beta)=\alpha^{\tilde{u}'}\beta$.
\end{lem}

\begin{proof}
(i) Suppose $\tau_1(\omega_d)=\alpha^{u_0}\beta^{v_0}$. Note that $u_0$ is odd: otherwise it is impossible for $\tau_1(\eta_1),\ldots,\tau_1(\eta_d),\tau_1(\omega_d)$ to generate $\Delta$.

Take $y$ with $yu_0\equiv'1-v_0$. Let $p=1+4zy$, and let
$$\tilde{u}=(1-2^{c-1}y)pu_0(1+2^{c-1}y(u_0-1)).$$
Let $\tau_2=\sigma^{(1-2^{c-1}y)p,y}_{0,1}$, so that $(\tau_2)_+=\sigma^{p,2y}_{0,1}$.
Then $(\tau_2\tau_1)(\omega_d)=\alpha^{\tilde{u}}\beta$, and by Lemma \ref{lem:conjugate}, $(\tau_2)_+\circ\sigma^{z,1}_{0,-z}=\sigma^{z,1}_{0,-z}\circ(\tau_2)_+$.

(ii) 
Take $y$ with $y\tilde{u}\equiv-2^{a-c}\overline{q}$, with $\overline{q}$ to be determined. Let $p'=1+2^{a-c}\overline{q}+4zy$.
Consider
\begin{align*}
u(\overline{q})&=(1-2^{c-1}y)p'\tilde{u}(1+2^{c-1}y(\tilde{u}-1))   \\
&=p'\tilde{u}\big((1-2^{c-1}y)2^{c-1}y\tilde{u}+1-2^cy+2^{2c-2}y^2\big)  \\
&\equiv^{(a)}p'\tilde{u}(-2^{a-1}\overline{q}+1)   \\
&\equiv^{(a)}\big(1+4zy+(1-2^{c-1}(1+4zy))2^{a-c}\overline{q}\big)\tilde{u}.
\end{align*}
Obviously, we can find $\overline{q}$ such that $u(\overline{q})\equiv^{(a)}\tilde{u}'$.

Let $\tau=\sigma^{p',0}_{0,1+2^{a-c}\overline{q}}$. Now $\tau_+=\sigma^{p',0}_{0,1+2^{a-c}\overline{q}}$ commutes with $\sigma^{z,1}_{0,-z}$ and $\tau(\alpha^{\tilde{u}}\beta)=\alpha^{\tilde{u}'}\beta$.
\end{proof}

Concluding from the above lemmas, up to isomorphism we may just assume $\varphi_+=\sigma^{z,1}_{0,-z}$ for a unique $z$ with $0\le z<2^{a-2}$ and $z\equiv^{(c-2)}-1$, and $\omega_d=\alpha^{\tilde{u}}\beta$ such that $\tilde{u}$ is an odd number whose residue modulo $2^{a-c}$ is unique.

\subsection{Expressing necessary and sufficient conditions in terms of congruence equations}

Remember that for each $k$,
$$(1+2^c)^k\equiv 1+2^ck, \qquad  [k]_{1+2^c}\equiv k(1+2^{c-1}(k-1)).$$
Implied by $z\equiv^{(c-2)}-1$,
\begin{align}
\|4(z+1)^2\|\ge 2c-2\ge a-1.   \label{eq:4-times}
\end{align}

Suppose $\eta_i=\alpha^{2u_i}\beta^{v_i}$.
Then $\omega_i\omega_d^{-1}=\eta_i\cdots\eta_1=\alpha^{2f_i}\beta^{g_i}$, where
\begin{align}
f_i&=u_i+(1+2^cv_i)u_{i-1}+\cdots+(1+2^c(v_i+\cdots+v_2))u_1, \label{eq:f} \\
g_i&=v_i+\cdots+v_1.   \label{eq:g}
\end{align}
So
$\omega_i=\alpha^{2f_i+(1+2^cg_i)\tilde{u}}\beta^{g_i+1}.$

The condition (\ref{eq:inverse}) is equivalent to
\begin{align}
f_{\ell+ti}+(1-2^c(g_i+1))f_i+(1+2^{c-1}(g_{\ell+ti}-1))\tilde{u}&\equiv 0, \label{eq:inverse-1} \\
g_{\ell+ti}+g_i+2&\equiv' 0.  \label{eq:inverse-2}
\end{align}
Also the condition $1=\omega_d\omega_d^{-1}=\alpha^{2f_d}\beta^{g_d}$ implies that
\begin{align}
f_d\equiv 0, \qquad g_d\equiv'0.   \label{eq:trivial}
\end{align}

From (\ref{eq:composite}) and $z^2\equiv^{(c-1)}1$ we see $\varphi_+^2=\sigma^{s,0}_{0,1}$, with
$s=z^2+2^{c-1}(z-1)$.
Hence
\begin{align}
u_{i+2}\equiv su_i,  \qquad  v_{i+2}\equiv'v_i.    \label{eq:2-step}
\end{align}
Clearly, $u_2\equiv u_1\pmod{2}$. It follows from (\ref{eq:Delta+}) that the $u_i$'s are all odd.

Put
$$\overline{u}=u_2+(1+2^c(u_1+v_1))u_1, \qquad  \overline{v}=v_2+v_1.$$
Since
$$u_2\equiv z[u_1]_{1+2^c}\equiv (z-2^{c-1}(u_1-1))u_1,  \qquad v_2\equiv'u_1-zv_1,$$
we have
\begin{align}
\overline{u}&\equiv \big(z+1+2^{c-1}(u_1+2v_1+1)\big)u_1, \label{eq:u-bar}  \\
(s-1)\overline{u}&\equiv (z+1)^2(z-1)u_1\stackrel{(\ref{eq:4-times})}\equiv 2(z+1)^2,  \label{eq:s-1-times-u-bar} \\
\overline{v}&\equiv'u_1+(1-z)v_1. \label{eq:v-bar}
\end{align}

Obviously, $s\equiv^{(c-1)}1$, so that for each $n$,
\begin{align}
s^n\equiv 1+n(s-1).    \label{eq:s-to-g}
\end{align}

Now (\ref{eq:f}), (\ref{eq:g}), (\ref{eq:2-step}) imply
\begin{align*}
f_{2k}&\equiv \overline{u}\cdot{\sum}_{j=0}^{k-1}(1+2^cj(u_1+2v_1))s^{k-1-j}\equiv \overline{u}[k]_s\equiv k\overline{u}-k(k-1)(z+1)^2, \\
f_{2k+1}&\equiv s^ku_1+(1+2^cv_1)f_{2k}\equiv (1+k(s-1))u_1+k\overline{u}-k(k-1)(z+1)^2, \\
g_{2k}&\equiv'k\overline{v} \qquad \text{and} \qquad g_{2k+1}\equiv'k\overline{v}+v_1.
\end{align*}

\begin{lem}
Let $h=(\ell-1)/2$. The conditions {\rm(\ref{eq:inverse-1})}--{\rm(\ref{eq:trivial})} hold if and only if
\begin{align}
h\overline{v}+v_1+2\equiv'0,  \label{eq:condition1} \\
u_1+h\overline{u}\equiv h(h+1)(z+1)^2+(3\cdot 2^{c-1}-1+h(s-1))\tilde{u}, \label{eq:condition2} \\
2(z+1)^2\equiv 2^{c-1}\overline{v}+(1-s), \label{eq:condition3} \\
\|t+1\|\ge\max\{a-c+2,b+1\},  \label{eq:condition4}  \\
\|d\|\ge\max\{a-c+2,b+1\}.  \label{eq:condition5}
\end{align}
\end{lem}

\begin{proof}
Let $e=(t+1)/2$.
Let (\ref{eq:inverse-2})$_{i=2k}$ stand for (\ref{eq:inverse-2}) when $i=2k$, and so forth.

The condition (\ref{eq:inverse-2})$_{i=2k}$ reads
\begin{align*}
(h+kt)\overline{v}+v_1+k\overline{v}+2\equiv'0,
\end{align*}
which, due to $\|t+1\|\ge b+1$, is equivalent to (\ref{eq:condition1}).
Conversely, if (\ref{eq:condition1}) is satisfied, then (\ref{eq:inverse-2})$_{i=2k+1}$ holds, too:
$$g_{\ell+t(2k+1)}+g_{2k+1}+2\equiv'(h+kt+e)\overline{v}+k\overline{v}+v_1+2\equiv'0.$$

In virtue of $2^c\overline{u}\equiv 0$ and (\ref{eq:4-times}), the condition (\ref{eq:inverse-1})$_{i=2k}$ reads
\begin{align}
s^{h+kt}u_1+h\overline{u}-(h^2-h+(2h+2)k)(z+1)^2 \nonumber \\
+\big(1+2^{c-1}((h+kt)\overline{v}+v_1-1)\big)\tilde{u}\equiv 0.  \label{eq:deduce-f1}
\end{align}
Then the difference between (\ref{eq:inverse-1})$_{i=2k+2}$ and (\ref{eq:inverse-1})$_{i=2k}$ is equal to
\begin{align}
(1-s)u_1+(2h+2)(z+1)^2-2^{c-1}\overline{v}\tilde{u}\equiv 0,  \label{eq:deduce-f2}
\end{align}
where (\ref{eq:4-times}), (\ref{eq:s-to-g}) have been used.

Setting $k=0$ in (\ref{eq:deduce-f1}), we obtain
\begin{align}
(1+h(s-1))u_1+h\overline{u}-h(h-1)(z+1)^2
+\big(1+2^{c-1}(h\overline{v}+v_1-1)\big)\tilde{u}\equiv 0. \label{eq:deduce-f3}
\end{align}

Clearly, (\ref{eq:inverse-1})$_{i=2k}$ holds for all $k$ if and only if (\ref{eq:deduce-f2}) and (\ref{eq:deduce-f3}) hold.
With (\ref{eq:s-1-times-u-bar}) referred to, (\ref{eq:deduce-f2}), (\ref{eq:deduce-f3}) are equivalent to
\begin{align}
u_1+h\overline{u}&\equiv h(h-1)(z+1)^2-(1+2^{c-1}(v_1-1))\tilde{u},  \label{eq:u1}   \\
2(z+1)^2&\equiv (2^{c-1}\overline{v}+(1-s))\tilde{u}.  \nonumber
\end{align}
Note that the second equation is equivalent to (\ref{eq:condition3}) and forces $\|s-1\|=c-1$. Hence $\|z+1\|=c-2$,
and by (\ref{eq:u-bar}), $\|\overline{u}\|=c-2$.

The condition (\ref{eq:inverse-1})$_{i=2k+1}$ reads
\begin{align}
(h+e)\overline{u}-(h^2-h-2(h+1)k)(z+1)^2+s^ku_1-2^c(k\overline{v}+v_1+1)u_1 \nonumber \\
+\big(1+2^{c-1}((h+kt)\overline{v}-1)\big)\tilde{u}\equiv 0.  \label{eq:deduce-f4}
\end{align}
So the difference between (\ref{eq:inverse-1})$_{i=2k+3}$ and (\ref{eq:inverse-1})$_{i=2k+1}$ equals
\begin{align*}
(s-1-2^c\overline{v})u_1+2(h+1)(z+1)^2-2^{c-1}\overline{v}\tilde{u}\equiv 0,
\end{align*}
which can be implied by (\ref{eq:u1}), assuming (\ref{eq:deduce-f2}).

Setting $k=0$ in (\ref{eq:deduce-f4}), we obtain
\begin{align*}
(h+e)\overline{u}-h(h-1)(z+1)^2+(1-2^c(v_1+1))u_1+(1+2^{c-1}(h\overline{v}-1))\tilde{u}\equiv 0;
\end{align*}
it combined with (\ref{eq:u1}) implies $e\overline{u}\equiv 0$, which is equivalent to (\ref{eq:condition4}).

By (\ref{eq:condition1}), (\ref{eq:condition3}), $2^{c-1}v_1\equiv h(1-s)-2h(z+1)^2-2^c$, and hence (\ref{eq:u1}) becomes (\ref{eq:condition2}).

Finally, (\ref{eq:trivial}) holds if and only if
$$\frac{d}{2}\overline{u}\equiv \frac{d}{2}\big(\frac{d}{2}-1\big)(z+1)^2, \qquad  \frac{d}{2}\overline{v}\equiv '0,$$
which are equivalent to (\ref{eq:condition5}), as is easy to verify.
\end{proof}

\medskip

Now since $\|z+1\|=c-2$ and $0\le z<2^{a-2}$, we may write
$$z=2^{c-2}(2x-1)-1, \qquad  1\le x\le 2^{a-c-1}.$$

By (\ref{eq:condition3}), using $(2x-1)^2\equiv1\pmod{4}$ and $c-1\ge b\ge a-c$, we obtain
\begin{align}
\overline{v}&\equiv^{(a-c)}\frac{z^2-1}{2^{c-1}}+z-1\equiv^{(a-c)}-2^{c-3}-2x-1, \nonumber  \\
v_1&\equiv^{(a-c)}-h\overline{v}-2\equiv^{(a-c)}-2^{c-3}h+h(2x+1)-2, \label{eq:v1-2} \\
u_1&\stackrel{(\ref{eq:v-bar})}{\equiv'}\overline{v}+(z-1)v_1\equiv^{(a-c)}4-2^{c-3}-(2h+1)(2x+1). \label{eq:u1-2}
\end{align}
By (\ref{eq:u-bar}), $\|\overline{u}\|=c-2$. By (\ref{eq:condition2}), $\tilde{u}\equiv^{(c-1)}h\overline{u}-u_1\equiv^{(c-1)}2^{c-2}h-u_1$, so that
\begin{align*}
\tilde{u}\equiv^{(a-c)}(2h+1)(2^{c-3}+2x+1)-4.
\end{align*}
According to the conclusion in the end of Section 3.2 we may just set
\begin{align*}
\tilde{u}=(2h+1)(2^{c-3}+2x+1)-4.
\end{align*}

By (\ref{eq:v1-2}), (\ref{eq:u1-2}),
\begin{align*}
u_1+2v_1+1&\equiv^{(a-c)}-2^{c-3}(2h+1)-2x.
\end{align*}
Hence by (\ref{eq:u-bar}),
\begin{align}
\overline{u}\equiv(z+1)u_1-2^{c-1}\big(2^{c-3}(2h+1)+2x\big)\tilde{u}.  \label{eq:u-bar-last}
\end{align}
Using
$$s-1=z^2-1+2^{c-1}(z-1)\equiv -2^{2c-4}-2^{c-1}(2x+1),$$
we convert (\ref{eq:condition2}) into
\begin{align*}
(1+2^{c-2}h(2x-1))u_1&\equiv 2^{c-1}h(2^{c-3}(2h+1)+2x)\tilde{u}+h(h+1)2^{2c-4} \\
&\ \ \ \ +\big(3\cdot 2^{c-1}-1-h(2^{2c-4}+2^{c-1}(2x+1))\big)\tilde{u}  \\
&\equiv h(h+1)2^{2c-4}+(2^{2c-3}h^2+(3-h)2^{c-1}-1)\tilde{u},
\end{align*}
implying
\begin{align}
u_1&\equiv (1-2^{c-2}h(2x-1)-2^{2c-4}h^2)\big(h(h+1)2^{2c-4}+(2^{2c-3}h^2+(3-h)2^{c-1}-1)\tilde{u}\big) \nonumber  \\
&\equiv h(h+1)2^{2c-4}+((3-h+hx)2^{c-1}-1+2^{c-2}h-2^{2c-4}h^2)\tilde{u}. \label{eq:u1-final}
\end{align}
So (\ref{eq:u-bar-last}) becomes
$$\overline{u}\equiv(2^{2c-4}(h+1)-2^{c-2}(6x-1))\tilde{u}.$$

Finally, (\ref{eq:u1-final}) implies
$$u_1\equiv'(2^{c-2}h-1)\tilde{u}\equiv '4-2^{c-3}-(2h+1)(2x+1).$$
Hence by (\ref{eq:v-bar}) and (\ref{eq:condition1}),
\begin{align*}
v_1\equiv'-\frac{hu_1+2}{1+h(1-z)}\equiv'\frac{h(1-2^{c-2}h)\tilde{u}-2}{2h+1-2^{c-2}h(2x-1)}
\equiv'h(2^{c-3}+2x+1)-2,
\end{align*}
where the meanings of fractions are self-evident. So
$$\overline{v}\equiv' u_1+(2+2^{c-2})v_1\equiv '-2^{c-3}-2x-1.$$

\subsection{The result}

Recall
$$h=\frac{\ell-1}{2}=\frac{1}{2}\Big(\frac{(t-1,d)}{2}-1\Big).$$

For each $x$ with $1\le x\le 2^{a-c-1}$, let
\begin{align*}
\tilde{u}&=(2h+1)(2^{c-3}+2x+1)-4,   \\
\overline{u}&=(2^{2c-4}(h+1)-2^{c-2}(6x-1))\tilde{u},  \\
u&=h(h+1)2^{2c-4}+((3-h+hx)2^{c-1}-1+2^{c-2}h-2^{2c-4}h^2)\tilde{u},  \\
f_{2k}&=k\overline{u}-k(k-1)2^{2c-4},  \\
f_{2k+1}&=(1+2^ck)u+k\overline{u}-k(k-1)2^{2c-4},   \\
g_{2k}&=-k(2^{c-3}+2x+1),  \\
g_{2k+1}&=(h-k)(2^{c-3}+2x+1)-2,
\end{align*}
and put $\mathcal{M}(x)=\mathcal{CM}(\Delta,\{\omega_1,\ldots,\omega_d\})$ with $\omega_i=\alpha^{2f_i+(1+2^cg_i)\tilde{u}}\beta^{g_i+1}.$

\begin{thm}\label{thm:main}
If $\Delta$ admits $d$-valent RBCM$_t$'s, then necessarily $\|d\|,\|t+1\|\ge\max\{a-c+2,b+1\}$ and $c>b$.
When these hold, each $d$-valent RBCM$_t$ on $\Delta$ has type I and is isomorphic to $\mathcal{M}(x)$ for a unique $x$ with $1\le x\le 2^{a-c-1}$.
\end{thm}

\ \\
Haimiao Chen \ \ \ \ chenhm$@$math.pku.edu.cn \ \ \ \ orcid: 0000-0001-8194-1264  \\
Jingrui Zhang \ \ \ \ nanfangzjr$@$163.com   \\
Department of Mathematics, \\
Beijing Technology and Business University, \\
100048, 11\# Fucheng Road, Haidian District, Beijing, China.

\end{document}